\documentclass[11pt]{article}
\usepackage[a4paper,margin=28mm]{geometry}
\usepackage[T1]{fontenc}
\usepackage{lmodern,amsmath,amssymb,amsthm,mathtools,bm,microtype}
\usepackage[numbers,sort&compress]{natbib}
\usepackage{xcolor,enumitem,needspace,booktabs}
\usepackage[colorlinks=true,linkcolor=blue!50!black,citecolor=blue!50!black,urlcolor=blue!50!black]{hyperref}
\hypersetup{pdftitle={Global Sobolev Convergence of Howard Iteration for Finite-Horizon Controlled Diffusions},pdfauthor={Linyu Miao}}
\newtheorem{theorem}{Theorem}[section]
\newtheorem{lemma}[theorem]{Lemma}
\newtheorem{proposition}[theorem]{Proposition}
\newtheorem{corollary}[theorem]{Corollary}
\newtheorem{example}[theorem]{Example}
\theoremstyle{definition}
\newtheorem{assumption}[theorem]{Assumption}
\theoremstyle{remark}
\newtheorem{remark}[theorem]{Remark}
\newcommand{\R}{\mathbb R}
\newcommand{\E}{\mathbb E}
\newcommand{\KL}{\operatorname{KL}}
\newcommand{\TV}{\operatorname{TV}}
\newcommand{\norm}[1]{\left\lVert#1\right\rVert}
\newcommand{\dd}{\,\mathrm d}
\providecommand{\doi}[1]{\url{https://doi.org/#1}}
\title{Global Sobolev Convergence of Howard Iteration\\for Finite-Horizon Controlled Diffusions}
\author{Linyu Miao\\[0.4em]
\small School of Mathematical Sciences, Nankai University,\\
\small Tianjin 300071, China\\[0.2em]
\small E-mail: \texttt{mly202210@126.com}}
\date{}
\begin{document}
\maketitle
\begin{abstract}
We study Howard policy iteration for finite-horizon stochastic control problems in which the control enters the diffusion coefficient. Each iteration selects a policy by maximizing the Hamiltonian and evaluates it through a linear parabolic equation. Since the improvement depends on the Hessian, even smooth data can produce discontinuous evaluation coefficients. We prove global convergence in the full parabolic Sobolev space under assumptions that allow this loss of regularity. In one space dimension, bounded measurable coefficients and uniform ellipticity suffice, without a large discount, a short horizon, small diffusion variation, or regularity assumptions on the improving policies. Multidimensional extensions hold under structural conditions on the diffusion matrix. The proof combines monotonicity with a parabolic growth estimate to show that the Bellman residual vanishes in measure. Uniform higher integrability then yields strong convergence of the values, gradients, Hessians, and time derivatives, and constructs the unique strong solution. We also establish residual error bounds, extend the argument to entropy regularization and zero-temperature limits, and obtain quadratic convergence for special one-dimensional models.
\end{abstract}

\section{Introduction}
Finite-horizon stochastic control seeks to optimize an expected running and terminal reward by choosing the dynamics of a diffusion process. When the control affects both the drift and the covariance, dynamic programming leads to a fully nonlinear parabolic Hamilton--Jacobi--Bellman equation. Approximating its solution is a basic task in computing the value of the control problem. We consider problems on bounded state domains, with stopping at exit or maturity, and study the convergence of Howard's policy iteration when the diffusion is controlled.

Howard's method \cite{Howard1960} alternates policy improvement by pointwise maximization of the Hamiltonian with policy evaluation through a linear equation. For controlled diffusions, these evaluations are linear second-order parabolic equations. The connection with Newton's method was developed by Puterman and Brumelle \cite{PutermanBrumelle1979} and Puterman \cite{Puterman1981}. Convergence is well understood for many discrete Bellman systems; see, for example, Bokanowski, Maroso, and Zidani \cite{BokanowskiMarosoZidani2009}. In continuous time and space, diffusion control creates an additional difficulty: the improving policy depends on the Hessian, so smooth model data and a smooth initial guess can produce discontinuous coefficients in the next evaluation. A convergence argument must therefore accommodate the low regularity generated by the iteration itself.

When the diffusion is uncontrolled, policy improvement changes only the lower-order terms of the evaluation equation. For finite-horizon problems in this setting, Kerimkulov, \v Si\v ska, and Szpruch \cite{Kerimkulov2020} prove exponential convergence and stability. For entropy-regularized control, Huang, Wang, and Zhou \cite{HuangWangZhou2025} establish convergence with uncontrolled diffusion, while Ma, Wang, and Zhang \cite{MaWangZhang2026} obtain super-exponential rates for finite-horizon drift control.

For controlled diffusion, Jacka and Mijatovi\'c \cite{JackaMijatovic2017} discuss finite-horizon controlled diffusions within a general policy improvement framework. Their diffusion assumptions require improving selectors to be uniformly Lipschitz on compact sets, a condition not guaranteed by smooth model data alone. For entropy-regularized iteration, Tran, Wang, and Zhang \cite{TranWangZhang2025} prove convergence for stationary equations with a sufficiently large discount and a small controlled perturbation of the covariance. Their argument establishes uniform classical estimates for the iterates; the finite-horizon extension discussed there retains this regularity approach and does not resolve convergence without diffusion smallness. Ma, Wang, and Zhang \cite{MaWangZhang2026} also treat scalar diffusion control, but their result is for an infinite horizon.

The remaining question is whether the original finite-horizon iteration converges when its own policy updates produce merely measurable coefficients. To our knowledge, global strong convergence without diffusion smallness or regularity assumptions on the improving policies has remained open even in one space dimension. We address this question under bounded measurable coefficients and uniform ellipticity.

We prove convergence in the full parabolic Sobolev space $W^{2,1}_2$ on a smooth bounded domain, for every finite horizon and every smooth initial guess with the prescribed traces. In one dimension, the second-order coefficient may be any bounded uniformly elliptic controlled coefficient; no large discount or small diffusion variation is required. Higher-dimensional results cover matrices proportional to a common smooth matrix field and small anisotropic perturbations of scalar matrices. The iteration constructs the unique strong solution. We then extend the convergence argument to relative-entropy regularization, including strong zero-temperature limits, and prove quadratic rates under additional one-dimensional assumptions. Strong-solution theory for related fully nonlinear equations is already available under Cordes conditions \cite{SmearsSuli2014,SmearsSuli2016}; the main contribution here is convergence of the Howard sequence itself. The proof combines a parabolic growth estimate \cite{Lin2015} with uniform higher integrability to obtain strong convergence of the Bellman residual, and then uses linear stability to recover convergence in the full Sobolev norm.

Section~\ref{sec:setting} introduces the control problem, Howard iteration, and the main theorem. Sections~\ref{sec:linear} and \ref{sec:convergence} give the linear estimates and convergence proof. Entropy regularization is treated separately in Section~\ref{sec:entropy}. Sections~\ref{sec:quadratic} and \ref{sec:examples} present quadratic rates and examples. Section~\ref{sec:scope} concludes with the scope of the results and further questions.

\section{Stochastic Control and Howard Iteration}\label{sec:setting}
Let $T>0$ be a fixed maturity and let $\Omega\subset\R^d$ be a bounded connected domain with smooth boundary. Fix a terminal reward $g\in C^\infty(\overline\Omega)$ with $g=0$ on $\partial\Omega$. To describe the control interpretation, use a weak formulation. An admissible system in $\mathcal A(t,x)$ consists of a filtered probability space
\[
 (\Omega_0,\mathcal F,\mathbb F,\mathbb P),\qquad
 \mathbb F=(\mathcal F_s)_{t\le s\le T},
\]
satisfying the usual completeness and right-continuity conditions, an $\mathbb F$-Brownian motion $W$, an $\mathbb F$-progressively measurable $U$-valued control process $\alpha$, and a continuous $\mathbb F$-adapted state process $X$ satisfying
\[
 dX_s=\widehat b^{\alpha_s}(s,X_s)\dd s
       +\widehat\sigma^{\alpha_s}(s,X_s)\dd W_s,\qquad X_t=x\in\Omega,
\]
up to $\zeta=T\wedge\inf\{s\ge t:X_s\notin\Omega\}$, where $U$ is a compact action space. The probability space and the state process are part of the admissible system. The value function is
\[
 V(t,x)=\sup_{\mathfrak a\in\mathcal A(t,x)}\E^{\mathfrak a}\left[
 \int_t^\zeta D_{t,s}\widehat f^{\alpha_s}(s,X_s)\dd s
       +D_{t,\zeta}g(X_\zeta)\right],
\]
where
\[
 D_{t,s}=\exp\!\left(-\int_t^s\widehat c^{\alpha_r}(r,X_r)\dd r\right).
\]
Here $\E^{\mathfrak a}$ denotes expectation under the probability measure of the admissible system $\mathfrak a$. The boundedness assumptions below make the reward integrable. Early exit gives zero terminal reward. Whenever the control problem has the usual dynamic programming characterization, its associated Bellman equation has terminal datum $g$ and zero lateral datum. We study this equation directly, under the coefficient assumptions below.

Put $Q=(0,T)\times\Omega$ and use remaining time $\tau=T-t$. The coefficients $A,b,c,f$ below are the time reversals of $\tfrac12\widehat\sigma\widehat\sigma^\top,\widehat b,\widehat c,\widehat f$, respectively. Set
\[
 X_p=W^{1,p}(0,T;L^p(\Omega))\cap L^p(0,T;W^{2,p}(\Omega)).
\]
This is the usual parabolic space $W^{2,1}_p(Q)$, with spatial order two and time order one. We use the equivalent norm
\[
 \norm{w}_{X_p}=\norm{w}_p+\norm{Dw}_p+\norm{D^2w}_p+\norm{w_\tau}_p,
\]
where $\norm{\cdot}_p$ means the $L^p(Q)$ norm unless another domain is specified. Spatial derivatives are denoted by $D,D^2$. Matrix norms are Frobenius norms, and coefficient bounds are essential bounds in time and state, uniform in the action. Define
\[
 X_p^0=\{w\in X_p:w\in L^p(0,T;W^{1,p}_0(\Omega)),\ w(0)=0\}.
\]
The initial trace is taken in $L^p(\Omega)$; the standard stronger parabolic trace is then also zero. We write $\partial_pQ$ for the initial and lateral boundary. Extend $g$ by $\widetilde g(\tau,x)=g(x)$ and put $\mathcal X_p=\widetilde g+X_p^0$. A strong solution is a function in this affine space satisfying the equation almost everywhere.

\begin{assumption}[Coefficients and diffusion structure]\label{ass:coeff}
The action space $U$ is compact metrizable. The functions $A^\alpha,b^\alpha,c^\alpha,f^\alpha$ are measurable in $(\tau,x)$ and continuous in $\alpha$, with uniform bounds. The matrices $A^\alpha$ are symmetric and $0\le c^\alpha\le C_0$. In a diffusion control model, $A^\alpha$ is one half of the covariance matrix, that is, $A^\alpha=\frac12\sigma^\alpha(\sigma^\alpha)^\top$. One of the following structural hypotheses holds.
\begin{enumerate}[label=(\roman*),leftmargin=*]
\item\label{case:shape} There is a smooth, control-independent symmetric matrix field $S$ on $\overline Q$, with $s_0I\le S\le s_1I$, such that
\begin{equation}\label{eq:shape}
 A^\alpha=a^\alpha S,\qquad 0<\underline a\le a^\alpha\le\overline a<\infty.
\end{equation}
\item\label{case:perturb} The domain is convex and
\begin{equation}\label{eq:perturb}
 A^\alpha=a^\alpha I+E^\alpha,\qquad
 \underline a\le a^\alpha\le\overline a,\qquad
 \sup_{\tau,x,\alpha}|E^\alpha|_F\le e_0<\underline a.
\end{equation}
Here $E^\alpha$ is symmetric and $|\cdot|_F$ is the Frobenius norm.
\end{enumerate}
\end{assumption}

For $d=1$ and $S=1$, condition \eqref{eq:shape} imposes only the stated ellipticity bounds on the second-order coefficient. In higher dimensions it requires the matrices to be proportional to a common matrix field. Condition \eqref{eq:perturb} allows a small controlled non-scalar perturbation. All policies, mixtures, and approximations in \eqref{eq:shape} use the same fixed field $S$ and common structural bounds. Both classes are preserved by measurable pointwise mixtures and convex combinations.

For $z=(r,p,M)\in\R\times\R^d\times\mathbb S^d$ define
\begin{align}
 h(\tau,x,z,\alpha)&=A^\alpha:M+b^\alpha\cdot p-c^\alpha r+f^\alpha,\label{eq:h}\\
 H(\tau,x,z)&=\max_{\alpha\in U}h(\tau,x,z,\alpha).\label{eq:hzero}
\end{align}
We often suppress $(\tau,x)$ and write $Ju=(u,Du,D^2u)$. The Bellman equation is
\begin{equation}\label{eq:HJB}
 F(u):=u_\tau-H(Ju)=0\quad\hbox{a.e. in }Q,\qquad u\in\mathcal X_2.
\end{equation}
The initial condition at $\tau=0$ is the terminal reward at maturity $T$. The homogeneous lateral condition corresponds to stopping at exit with zero exit reward.

A policy in Howard iteration is a measurable map $\alpha:Q\to U$. In stochastic time, its feedback rule is $\alpha_s=\alpha(T-s,X_s)$. For such a map, write
\[
 \ell^\alpha w=A^\alpha:D^2w+b^\alpha\cdot Dw-c^\alpha w,
 \qquad L^\alpha=\partial_\tau-\ell^\alpha.
\]
Choose $u^0\in C^\infty(\overline Q)$ with the prescribed initial and lateral traces; one choice is $u^0=\widetilde g$. Given $u^n$, select a measurable maximizer
\[
 \alpha_n(\tau,x)\in\operatorname*{arg\,max}_{\alpha\in U}
 h(\tau,x,Ju^n(\tau,x),\alpha)
\]
and evaluate it by solving
\begin{equation}\label{eq:Howard}
 L^{\alpha_n}u^{n+1}=f^{\alpha_n},\qquad u^{n+1}\in\mathcal X_{p_*}.
\end{equation}
The exponent $p_*>2$ is constructed below independently of the policy. No subsolution condition is imposed on $u^0$.

\Needspace{9\baselineskip}
\begin{theorem}[Global strong convergence of Howard iteration]\label{thm:main}
Under Assumption~\ref{ass:coeff}, there exists $p_*>2$ with the following properties.
\begin{enumerate}[label=(\roman*),leftmargin=*]
\item Every step \eqref{eq:Howard} is uniquely solvable. Starting with the first evaluation,
\[
 u^1\le u^2\le\cdots,\qquad
 \sup_{n\ge0}\norm{u^n}_{X_{p_*}}+\sup_{n\ge1}\norm{u^n}_{L^\infty(Q)}<\infty.
\]
\item Equation \eqref{eq:HJB} has a unique solution $u$ in $\mathcal X_2$. It belongs to $X_{p_*}\cap L^\infty(Q)$, and
\begin{equation}\label{eq:strongconv}
 \norm{u^n-u}_{X_2}\longrightarrow0.
\end{equation}
In particular the time derivatives, gradients, and Hessians converge strongly in $L^2(Q)$.
\item With $R_n=H(Ju^n)-u^n_\tau$, for $n\ge1$,
\begin{equation}\label{eq:posterior}
 R_n\ge0,\qquad R_n\longrightarrow0\text{ in }L^2(Q),\qquad
 \norm{u^n-u}_{X_2}\le C\norm{R_n}_{L^2(Q)}.
\end{equation}
\end{enumerate}
The constants may depend on the data and the smooth initial guess, but are independent of the sequence of maximizing selections. No short-horizon or positive lower bound on the discount is required. In \eqref{eq:shape}, no smallness condition is imposed on $\overline a/\underline a$. Global convergence here means convergence from every such initial guess.
\end{theorem}

For $d=1$, parabolic Sobolev embedding also gives
\[
 \norm{u^n-u}_{C^{\beta/2,\beta}(\overline Q)}\longrightarrow0,
 \qquad 0<\beta<1/2.
\]

\section{Uniform Estimates for Linear Policy Evaluation}\label{sec:linear}
We first work with a linear operator
\begin{equation}\label{eq:linear}
 Lw=w_\tau-A:D^2w-b\cdot Dw+cw
\end{equation}
whose coefficients satisfy the bounds of Assumption~\ref{ass:coeff}. All statements in this section include arbitrary measurable mixtures of the allowed coefficients.

\begin{lemma}[Uniform inverse estimates for policy evaluation]\label{lem:inverse}
There are $p_*>2$ and $C_p<\infty$, depending only on the fixed coefficient bounds, the structural data, $T$, and $\Omega$, such that
\begin{equation}\label{eq:inverse}
\begin{gathered}
 L:X_p^0\longrightarrow L^p(Q)\text{ is an isomorphism},\\
 \norm{w}_{X_p}\le C_p\norm{Lw}_{L^p(Q)},\qquad 2\le p\le p_*.
\end{gathered}
\end{equation}
The constant \(C_p\) is uniform over all admissible policies, regardless of the roughness of the measurable coefficients induced by the policy.
\end{lemma}
\begin{proof}
First consider \eqref{eq:shape}. Put
\[
 a_*=(\underline a+\overline a)/2,\qquad
 d_*=(\overline a-\underline a)/2,\qquad B_\tau w=\operatorname{div}(S(\tau,\cdot)Dw).
\]
Then $L=\partial_\tau-aB_\tau-\beta\cdot D+c$, where $\beta=b-a\operatorname{div}S$ is uniformly bounded. For $\kappa>0$, let
\[
 L_{*,\kappa}=\partial_\tau-a_* B_\tau+\kappa,
 \qquad \mathcal R_\kappa=L_{*,\kappa}^{-1}:L^p(Q)\longrightarrow X_p^0.
\]
Standard Dirichlet solvability for the smooth reference operator gives, for each $1<p<\infty$, a unique $w\in X_p^0$ solving $L_{*,\kappa}w=r$ for $r\in L^p(Q)$, and
\[
 \norm{w}_{X_p}\le C(p,\kappa,S,\Omega,T)\norm{r}_p.
\]
One precise reference is \cite[Theorem 2.1, author version]{DenkHieberPruss2007}, applied with differential order $2m=2$, scalar unknown, Dirichlet boundary operator, and zero initial and boundary data; see also \cite[Chapter IV]{LSU1968}. The smooth coefficients satisfy the regularity assumptions of that theorem. Uniform positivity gives normal ellipticity and the scalar Dirichlet complementing condition. The boundary is smooth, and both the initial trace in $B^{2-2/p}_{p,p}(\Omega)$ and the lateral trace are zero, so compatibility is automatic.

Let $w=\mathcal R_\kappa r$. Testing by $-B_\tau w$, integrating in time, and using the zero traces yields
\begin{align}\label{eq:energy}
 &\frac12\int_\Omega Dw(T)^{\!\top}S(T)Dw(T)
 +a_*\norm{B_\tau w}_2^2
 +\kappa\int_Q Dw^{\!\top}SDw\notag\\
 &\hspace{13mm}-\frac12\int_Q Dw^{\!\top}S_\tau Dw
 =-\int_Q r B_\tau w.
\end{align}
Let $\kappa\ge1+\norm{S_\tau}_\infty/(2s_0)$. The two gradient terms on the left of \eqref{eq:energy} then have nonnegative sum, so
$a_*\norm{B_\tau w}_2^2\le\norm{r}_2\norm{B_\tau w}_2$.
Testing the reference equation separately with $w$ gives
\[
 a_*s_0\norm{Dw}_2^2+\kappa\norm{w}_2^2
 \le\norm{r}_2\norm{w}_2.
\]
It follows that $\norm{w}_2\le\kappa^{-1}\norm{r}_2$ and
$\norm{Dw}_2^2\le(a_*s_0\kappa)^{-1}\norm{r}_2^2$. Thus
\begin{equation}\label{eq:Rbounds}
 \norm{B_\tau\mathcal R_\kappa}_{2\to2}\le a_*^{-1},\quad
 \norm{D\mathcal R_\kappa}_{2\to2}\le C\kappa^{-1/2},\quad
 \norm{\mathcal R_\kappa}_{2\to2}\le\kappa^{-1}.
\end{equation}
The smooth elliptic Dirichlet estimate
\[
 \norm{w(\tau,\cdot)}_{W^{2,2}(\Omega)}
 \le C_S\big(\norm{B_\tau w(\tau,\cdot)}_{L^2(\Omega)}
                    +\norm{w(\tau,\cdot)}_{L^2(\Omega)}\big)
\]
has a constant uniform over $\tau\in[0,T]$. The equation then also controls $w_\tau$ in $L^2$. Identity \eqref{eq:energy} is justified first for smooth data and then by density; the continuous $H^1_0(\Omega)$ time trace of $X_2^0$ justifies its terminal energy term.

For a zero-trace function the substitution $v=e^{-\kappa\tau}w$ changes $Lw=r$ into $(L+\kappa)v=e^{-\kappa\tau}r$. Factor
\begin{equation}\label{eq:factor}
 (L+\kappa)\mathcal R_\kappa=I-\mathcal T,\qquad
 \mathcal T=[(a-a_*)B_\tau+\beta\cdot D-c]\mathcal R_\kappa.
\end{equation}
Consequently
\begin{equation}\label{eq:Tbound}
 \norm{\mathcal T}_{2\to2}
 \le \frac{d_*}{a_*}+C\norm{\beta}_\infty\kappa^{-1/2}
       +\norm{c}_\infty\kappa^{-1}=:\rho<1
\end{equation}
after increasing $\kappa$. The essential point is that $d_*/a_*<1$ for every positive $\underline a$, however large the finite contrast.

Fix this $\kappa$. Regularity of the smooth reference equation gives a uniform finite bound $\norm{\mathcal T}_{4\to4}\le K$, where $K\ge1$ is independent of the measurable coefficients $a,b,c$. The reference inverses at exponents $2$ and $4$ agree on $L^4(Q)\subset L^2(Q)$: the $X_4^0$ solution is an $X_2^0$ solution, where uniqueness holds. The Riesz--Thorin theorem therefore interpolates the same operator $\mathcal T$ at the two endpoints:
\[
 \norm{\mathcal T}_{p\to p}\le\rho^{1-\theta}K^\theta,
 \qquad \frac1p=\frac{1-\theta}{2}+\frac\theta4.
\]
Choose $0<\theta_*<1$ so small that $\rho^{1-\theta_*}K^{\theta_*}<1$, and set $p_*=4/(2-\theta_*)>2$. For $2\le p\le p_*$, put $\rho_p=\rho^{1-\theta}K^\theta<1$. Then
\[
 (L+\kappa)^{-1}=\mathcal R_\kappa\sum_{j=0}^{\infty}\mathcal T^j,
 \qquad
 \norm{(L+\kappa)^{-1}}_{L^p\to X_p^0}
 \le\frac{\norm{\mathcal R_\kappa}_{L^p\to X_p^0}}{1-\rho_p}.
\]
The series converges in operator norm and constructs the $X_p^0$ solution. If $(L+\kappa)v=0$, apply \eqref{eq:factor} to $L_{*,\kappa}v$ to obtain uniqueness. Undoing the time weight gives \eqref{eq:inverse}.

For \eqref{eq:perturb}, use $B=\Delta$. The Miranda--Talenti inequality on a convex domain gives $\norm{D^2w}_2\le\norm{\Delta w}_2$ for $w\in H^2\cap H^1_0$ (also used in the Cordes framework of \cite{SmearsSuli2014}). The extra contribution $E:D^2\mathcal R_\kappa$ has operator norm at most $e_0/a_*$. Thus the leading part of \eqref{eq:Tbound} is $(d_*+e_0)/a_*<1$, precisely because $e_0<\underline a$. The same interpolation argument applies.
\end{proof}

\begin{lemma}[Positivity and smooth approximation]\label{lem:positive}
Let $L$ satisfy the hypotheses of Lemma~\ref{lem:inverse}. Its inverse preserves nonnegativity:
\[
 r\in L^2(Q),\quad r\ge0\ \text{a.e.}
 \quad\Longrightarrow\quad L^{-1}r\ge0\ \text{a.e. in }Q.
\]
For $v,w\in X_2$, suppose that $v-w-G\in X_2^0$ for some $G\in C^\infty(\overline Q)$. Then
\[
 \left.
 \begin{aligned}
 Lv&\ge Lw &&\text{a.e. in }Q,\\
 G&\ge0 &&\text{on }\partial_pQ
 \end{aligned}
 \right\}
 \quad\Longrightarrow\quad v\ge w\ \text{a.e. in }Q.
\]
Let $L_j=\partial_\tau-A_j:D^2-b_j\cdot D+c_j$ have common structural bounds with $L$, using the same fixed field $S$ in case \eqref{eq:shape}. If
\[
 (A_j,b_j,c_j)\to(A,b,c)\quad\text{a.e. in }Q,
 \qquad r_j\to r\quad\text{in }L^2(Q),
\]
then $L_j^{-1}r_j\to L^{-1}r$ in $X_2$.
\end{lemma}
\begin{proof}
First set $w=L^{-1}r$ and $w_j=L_j^{-1}r_j$. Since $w_j-w\in X_2^0$,
\[
 L_j(w_j-w)=r_j-r+(L-L_j)w.
\]
Lemma~\ref{lem:inverse} gives
\[
 \norm{w_j-w}_{X_2}
 \le C\bigl(\norm{r_j-r}_2+\norm{(L-L_j)w}_2\bigr)\longrightarrow0.
\]
The coefficient term tends to zero by dominated convergence applied to the fixed $L^2$ derivatives of $w$.

For comparison, put $z=v-w$ and $r=Lz\ge0$. To approximate the coefficients, extend $a$ by a number in $[\underline a,\overline a]$, extend $b,c$ by zero, and, in case \eqref{eq:perturb}, extend $E$ by zero. Mollification preserves these convex bounds. In case \eqref{eq:shape}, keep $S$ fixed and smooth only $a,b,c$; in case \eqref{eq:perturb}, smooth $a,E,b,c$. Along a subsequence the smooth coefficients converge almost everywhere to those of $L$. Choose smooth $r_j\ge0$ with $r_j\to r$ in $L^2$, and solve
\[
 L_jz_j=r_j\quad\text{in }Q,\qquad z_j-G\in X_2^0.
\]
The maximum principle gives $z_j\ge0$, since $G\ge0$ on $\partial_pQ$. Also $z_j-z\in X_2^0$, so
\[
 \norm{z_j-z}_{X_2}
 \le C\bigl(\norm{r_j-r}_2+\norm{(L-L_j)z}_2\bigr)\longrightarrow0.
\]
Thus $v-w=z\ge0$ almost everywhere. Taking $v=L^{-1}r$ and $w=G=0$ gives positivity of the inverse.
\end{proof}

\begin{remark}[Perturbations of a fixed matrix field]
Lemma~\ref{lem:inverse} also permits $A=aS+E$ when
$\sup|E|_F\norm{D^2\mathcal R_\kappa}_{2\to2}<1-\rho$ in \eqref{eq:Tbound}, provided uniform ellipticity is retained. This yields a neighborhood depending on the reference inverse constant. Condition \eqref{eq:perturb}, by contrast, gives the explicit sufficient bound $e_0<\underline a$ on convex domains.
\end{remark}

\section{Convergence in Measure and Strong Convergence}\label{sec:convergence}
We use monotonicity and uniform bounds to show that the Bellman residual vanishes, and then apply linear stability to prove strong convergence.

\begin{lemma}[Uniform bounds and policy improvement]\label{lem:iterbounds}
All iterates are well defined in $\mathcal X_{p_*}$. There are constants independent of $n$ such that
\begin{equation}\label{eq:iterbounds}
\begin{gathered}
 \norm{u^n}_{X_{p_*}}\le M_p\quad(n\ge0),\\
 |u^n|\le M_\infty,\qquad \norm{R_n}_{p_*}\le M_R\quad(n\ge1).
\end{gathered}
\end{equation}
Moreover, for $n\ge1$,
\begin{equation}\label{eq:increment}
 R_n\ge0,\qquad
 L^{\alpha_n}(u^{n+1}-u^n)=R_n,\qquad u^{n+1}\ge u^n.
\end{equation}
\end{lemma}
\begin{proof}
Fix measurable representatives of $Ju^n$; compactness of $U$ and the Carath\'eodory property of $h$ give a measurable maximizing policy by the measurable maximum theorem. The selected coefficients retain the bounds of Assumption~\ref{ass:coeff}. Subtracting $\widetilde g$ and applying Lemma~\ref{lem:inverse} gives
\[
 \norm{u^{n+1}}_{X_{p_*}}
 \le C\bigl(\norm{f^{\alpha_n}}_{p_*}
             +\norm{L^{\alpha_n}\widetilde g}_{p_*}
             +\norm{\widetilde g}_{X_{p_*}}\bigr)\le C.
\]
This proves existence and the uniform Sobolev bound by induction. Put $M_f=\sup_\alpha\norm{f^\alpha}_\infty$. The smooth barriers
\[
 \underline v(\tau)=-\norm{g}_\infty-M_f\tau,\qquad
 \overline v(\tau)=\norm{g}_\infty+M_f\tau
\]
have ordered parabolic traces and satisfy
$L^{\alpha_n}\underline v\le-M_f\le f^{\alpha_n}
\le M_f\le L^{\alpha_n}\overline v$, since $c^\alpha\ge0$.
Lemma~\ref{lem:positive} gives the common $L^\infty$ bound.

For $n\ge1$, the previous evaluation and the new maximization imply
\[
 u^n_\tau=\ell^{\alpha_{n-1}}u^n+f^{\alpha_{n-1}}
 \le H(Ju^n)=\ell^{\alpha_n}u^n+f^{\alpha_n}.
\]
The definition of $R_n$ and the maximizing property of $\alpha_n$ give
\[
 L^{\alpha_n}u^n=f^{\alpha_n}-R_n.
\]
Subtracting this identity from the evaluation equation for $u^{n+1}$ yields \eqref{eq:increment}, and positivity of the inverse gives monotonicity. Finally, $|H(z)|\le C(1+|z|)$ gives the residual bound from the Sobolev bound.
\end{proof}

\begin{lemma}[Uniform growth from a measurable source]\label{lem:growth}
Fix a cylinder $Q'=(s-r^2,s)\times B_r(x_0)$ compactly contained in $Q$, a number $0<m_0\le|Q'|$, and $\gamma>0$. Let $A$ belong to either matrix class in Assumption~\ref{ass:coeff}. For a measurable $\Gamma\subset Q'$ with $|\Gamma|\ge m_0$, let
\[
 P_Av:=v_\tau-A:D^2v=\gamma\mathbf1_\Gamma\text{ in }Q,\qquad v\in X_2^0.
\]
There is a positive constant $c_*=c_*(Q',m_0,\underline\nu,\overline\nu,d)$, independent of $A$ and $\Gamma$, such that
\begin{equation}\label{eq:growth}
 v\ge0\ \text{a.e.},\qquad \norm{v}_{L^2(Q)}\ge c_*\gamma.
\end{equation}
Here $\underline\nu,\overline\nu$ are common ellipticity bounds for $A$.
\end{lemma}
\begin{proof}
For smooth coefficients and a nonnegative smooth source, positivity follows from the maximum principle. A solution satisfies
\[
 v_\tau-\mathcal P^-(D^2v)\ge P_Av,
\]
where
\[
 \mathcal P^-(M)=\underline\nu\sum_{e_i>0}e_i
                    +\overline\nu\sum_{e_i<0}e_i
\]
and $e_i$ are the eigenvalues of $M$. The growth estimate \cite[Theorem 1.3, author version]{Lin2015} applies to nonnegative supersolutions; it does not require zero boundary values on the smaller cylinder. In the unit cylinder, a source exceeding a level on a fraction $m$ of the cylinder gives a lower bound $c m^\rho\exp(-\beta/m^2)$ times that level, for $|x|\le\kappa$ and $-\kappa m\le\tau\le0$, where $0<\kappa<1$ is fixed.

We record explicitly how to pass this estimate to the rough equation. Approximate the coefficients as in Lemma~\ref{lem:positive}, and choose smooth $f_j$ such that $0\le f_j\le\gamma$ and $f_j\to\gamma\mathbf1_\Gamma$ in $L^2(Q)$. On $\Gamma$,
\[
 |\{f_j\le\gamma/2\}\cap\Gamma|
 \le\frac4{\gamma^2}\norm{f_j-\gamma\mathbf1_\Gamma}_2^2.
\]
Thus, for large $j$, the fraction $m_j:=|\{f_j>\gamma/2\}\cap Q'|/|Q'|$ satisfies $m_j\ge m_*:=m_0/(2|Q'|)>0$. Let $v_j$ be the nonnegative global zero-trace solutions. Rescale $Q'$ to the unit cylinder; the source level becomes $r^2\gamma/2$. Choose $\kappa=1/2$. The lower-bound region for $m_j$ contains the region for $m_*$, and $m^\rho\exp(-\beta/m^2)$ is increasing for $m>0$. Hence all sufficiently large $j$ satisfy
\begin{equation}\label{eq:fixedgrowthregion}
 v_j\ge \frac{c r^2\gamma}{2}\,
             m_*^\rho e^{-\beta/m_*^2}
 \quad\hbox{on }\
 Y=(s-r^2m_*/2,s)\times B_{r/2}(x_0).
\end{equation}
The set $Y$ has fixed positive measure and is independent of the approximating coefficients and sources. Lemma~\ref{lem:positive} gives $v_j\to v$ in $X_2$. Passing to an almost-everywhere convergent subsequence proves \eqref{eq:fixedgrowthregion} for $v$ almost everywhere. Its $L^2(Y)$ norm gives \eqref{eq:growth}.
\end{proof}

\begin{remark}[Direct use of the growth estimate in one dimension]
For $d=1$, the growth estimate can also be applied directly to $X_2$ solutions, which have continuous representatives, using Krylov's result for measurable coefficients and $W^{2,1}_{d+1,\mathrm{loc}}\cap C$ supersolutions.\footnote{N.~V.~Krylov, \emph{On a result of Jessica Lin}, preprint (2013), Theorem~2.1, \url{https://arxiv.org/abs/1310.8322v2}.} The approximation in Lemma~\ref{lem:growth} covers the higher-dimensional case, where the available integrability need not reach $d+1$.
\end{remark}

\begin{lemma}[Vanishing responses and convergence in measure]\label{lem:residual}
Let $L_n$ have the common coefficient bounds and structure of Lemma~\ref{lem:inverse}. Suppose
\[
 L_n\delta_n=R_n\ge0,\qquad \delta_n\in X_2^0,\qquad
 \sup_n\norm{\delta_n}_{X_2}<\infty,\qquad \norm{\delta_n}_2\to0.
\]
Then $R_n\to0$ in measure on $Q$. If in addition $\sup_n\norm{R_n}_{p_*}\le M_R$ for some $p_*>2$, then $R_n\to0$ in $L^2(Q)$.
\end{lemma}
\begin{proof}
Positivity gives $\delta_n\ge0$. The Hessians are bounded in $L^2$, so integration by parts gives
\begin{equation}\label{eq:gradientincrement}
 \norm{D\delta_n}_2^2
 =-\int_Q\delta_n\Delta\delta_n
 \le\norm{\delta_n}_2\norm{\Delta\delta_n}_2\longrightarrow0.
\end{equation}
Write $L_n=\partial_\tau-A_n:D^2-b_n\cdot D+c_n$. The equation becomes
\[
 P_{A_n}\delta_n=R_n+\eta_n,\qquad
 \eta_n=b_n\cdot D\delta_n-c_n\delta_n\longrightarrow0\text{ in }L^2.
\]
Let $P_{A_n}w_n=|\eta_n|$, with zero traces. Then $w_n\ge0$ and $\norm{w_n}_{X_2}\to0$. Consequently
\begin{equation}\label{eq:corrected}
 z_n:=\delta_n+w_n\ge0,\qquad
 P_{A_n}z_n\ge R_n,\qquad \norm{z_n}_2\longrightarrow0.
\end{equation}

If convergence in measure failed, there would be $\gamma,m>0$ and a subsequence for which $|\{R_n>\gamma\}|\ge m$. Choose a compact set $K\subset Q$ with $|Q\setminus K|<m/2$, and cover $K$ by $N$ cylinders $Q_i$ compactly contained in $Q$. For each index in the subsequence, at least one cylinder has $|\{R_n>\gamma\}\cap Q_i|\ge m/(2N)$. Pass to a further subsequence for which this cylinder is fixed, and write $\Gamma_n=\{R_n>\gamma\}\cap Q'$. Compare $z_n$ with the zero-trace solution of $P_{A_n}v_n=\gamma\mathbf1_{\Gamma_n}$ using Lemma~\ref{lem:positive}. Lemma~\ref{lem:growth} gives $\norm{z_n}_2\ge\norm{v_n}_2\ge c_*\gamma$, contradicting \eqref{eq:corrected}. This proves convergence in measure on the whole cylinder.

The last passage to strong $L^2$ convergence can also be seen directly. For $M>0$,
\[
 \int_{\{R_n>M\}}R_n^2\le M^{2-p_*}M_R^{p_*}.
\]
On $\{R_n\le M\}$, convergence in measure and the boundedness by $M$ give convergence to zero of the integral of $R_n^2$. Let $n\to\infty$ and then $M\to\infty$. The strict inequality $p_*>2$ is essential in this step.
\end{proof}

\begin{lemma}[Linearization and stability]\label{lem:secant}
For $v,w\in\mathcal X_2$, there is a measurable linear operator $\widehat L$ in the convexly enlarged coefficient class of Lemma~\ref{lem:inverse} such that
\begin{equation}\label{eq:secant}
 F(v)-F(w)=\widehat L(v-w).
\end{equation}
In particular,
\begin{equation}\label{eq:nonlinearstability}
 \norm{v-w}_{X_2}\le C\norm{F(v)-F(w)}_2.
\end{equation}
The constant depends only on the common linear bounds.
\end{lemma}
\begin{proof}
Choose maximizing controls $\alpha_v,\alpha_w$ at the two jets. Comparing each maximum with the other selected control gives
\begin{equation}\label{eq:secantbounds}
 \ell^{\alpha_w}(v-w)\le H(Jv)-H(Jw)
                       \le\ell^{\alpha_v}(v-w).
\end{equation}
Set $q_-=\ell^{\alpha_w}(v-w)$, $q_+=\ell^{\alpha_v}(v-w)$, and $q=H(Jv)-H(Jw)$. Define the measurable weight
\[
 \theta=
 \begin{cases}(q-q_-)/(q_+-q_-),&q_+>q_-,\\0,&q_+=q_-.\end{cases}
\]
Then $0\le\theta\le1$ and $q=\theta q_++(1-\theta)q_-$. Use this same weight for the complete coefficient triples:
\[
 (\widehat A,\widehat b,\widehat c)
 =\theta(A^{\alpha_v},b^{\alpha_v},c^{\alpha_v})
       +(1-\theta)(A^{\alpha_w},b^{\alpha_w},c^{\alpha_w}).
\]
The resulting $\widehat L=\partial_\tau-\widehat A:D^2-\widehat b\cdot D+\widehat c$ satisfies \eqref{eq:secant}. The convex combination preserves the structural bounds. Lemma~\ref{lem:inverse} now gives \eqref{eq:nonlinearstability}.
\end{proof}

\begin{proof}[Proof of Theorem~\ref{thm:main}]
Lemma~\ref{lem:iterbounds} gives existence, uniform bounds, and monotonicity from the first evaluation. Dominated convergence shows that $u^n$ converges in $L^2$, so $\delta_n=u^{n+1}-u^n\to0$ in $L^2$. The uniform $X_{p_*}$ bound also bounds $\delta_n$ in $X_2$. Lemma~\ref{lem:residual}, applied to \eqref{eq:increment}, gives $R_n\to0$ in measure and in $L^2$. Lemma~\ref{lem:secant} now yields
\[
 \norm{u^n-u^m}_{X_2}\le C\norm{R_n-R_m}_2\longrightarrow0.
\]
Thus the sequence has a limit $u\in\mathcal X_2$. Global Lipschitz continuity of the Hamiltonian in the jet and $R_n\to0$ show that $F(u)=0$ in $L^2$. The uniform $X_{p_*}$ bound gives the same limit in $X_{p_*}$ by weak compactness and uniqueness of the distributional limit. The $L^\infty$ bound passes to the limit. Applying \eqref{eq:nonlinearstability} to two solutions gives uniqueness in $\mathcal X_2$, and applying it to $u^n,u$ proves \eqref{eq:posterior}.
\end{proof}

\section{Entropy Regularization}\label{sec:entropy}
We now extend the preceding argument to relative-entropy regularization. The additional estimate needed for convergence at fixed temperature is a bound on the entropy penalty of the improving policy. For related exploratory HJB equations, see \cite{TangZhangZhou2022}; finite-horizon policy improvement with uncontrolled diffusion is also studied by Tang and Zhou.\footnote{W.~Tang and X.~Y.~Zhou, \emph{Regret of exploratory policy improvement and $q$-learning}, preprint, arXiv:2411.01302v2 (2026), \url{https://arxiv.org/abs/2411.01302v2}.}

\subsection{The Regularized Model and Convergence Theorem}
Retain the coefficients, domain, and traces of Section~\ref{sec:setting}. Write $H_0=H$, $F_0=F$, and denote the unregularized solution by $u_0$. Fix a full-support probability measure $\mu$ on $U$. For $\lambda>0$, set
\begin{equation}\label{eq:hentropy}
 H_\lambda(\tau,x,z)=\lambda\log\int_U e^{h(\tau,x,z,\alpha)/\lambda}\,\mu(\dd\alpha),
\end{equation}
and consider
\begin{equation}\label{eq:entropyHJB}
 F_\lambda(u):=u_\tau-H_\lambda(Ju)=0,\qquad u\in\mathcal X_2.
\end{equation}
The normalization of $\mu$ fixes the relative-entropy convention. An unnormalized measure would add a running reward constant, which can affect a controlled exit problem.

\begin{assumption}[Actions for positive temperature]\label{ass:actions}
For $\lambda>0$ assume either that $U$ is finite with $\mu(\{\alpha\})>0$ for every action, or that $U\subset\R^m$ is compact, the coefficients in \eqref{eq:h} are uniformly Lipschitz in the control variable, and
\begin{equation}\label{eq:mass}
 \mu(B_r(\alpha)\cap U)\ge c_\mu r^m\quad
 (\alpha\in U,\ 0<r\le r_\mu)
\end{equation}
for some positive constants. Assumption~\ref{ass:actions} is not needed for $\lambda=0$.
\end{assumption}

For a measurable probability kernel $\pi=\pi(\tau,x;\dd\alpha)$ on $U$, define
\[
 (A^\pi,b^\pi,c^\pi,f^\pi)(\tau,x)
 =\int_U(A^\alpha,b^\alpha,c^\alpha,f^\alpha)(\tau,x)\,
                  \pi(\tau,x;\dd\alpha),
\]
and write
\[
 \ell^\pi w=A^\pi:D^2w+b^\pi\cdot Dw-c^\pi w,\qquad
 L^\pi=\partial_\tau-\ell^\pi.
\]
For positive temperature let $k_\lambda(\pi)=\lambda\KL(\pi\|\mu)$, where $\KL(\pi\|\mu)=\int\log(\dd\pi/\dd\mu)\dd\pi$ for $\pi\ll\mu$ and is $+\infty$ otherwise. At zero temperature we optimize over pure policies and set $k_0=0$.

For every $\lambda>0$, use the same initial guess $u^0_\lambda=u^0$. The improved policy is the Gibbs kernel
\begin{equation}\label{eq:gibbs}
 \frac{\dd\pi_{n,\lambda}}{\dd\mu}(\alpha)
 =\frac{\exp(h(Ju^n_\lambda,\alpha)/\lambda)}
        {\int_U\exp(h(Ju^n_\lambda,\beta)/\lambda)\,\mu(\dd\beta)}.
\end{equation}
The next evaluation solves
\begin{equation}\label{eq:entropyHoward}
 L^{\pi_{n,\lambda}}u^{n+1}_\lambda
 =f^{\pi_{n,\lambda}}-k_\lambda(\pi_{n,\lambda}),
 \qquad u^{n+1}_\lambda\in\mathcal X_{p_*}.
\end{equation}
For fixed $\lambda$, we suppress it in $u^n_\lambda$ and $\pi_{n,\lambda}$. At $\lambda=0$, interpret $\pi_{n,0}=\delta_{\alpha_n}$ and recover Section~\ref{sec:setting}.

\begin{theorem}[Entropy-regularized Howard iteration]\label{thm:entropy}
Under Assumptions~\ref{ass:coeff} and \ref{ass:actions}, for each $\lambda>0$ the conclusions of Theorem~\ref{thm:main} hold for \eqref{eq:entropyHJB} and \eqref{eq:entropyHoward}, with residual
$R_{n,\lambda}=H_\lambda(Ju^n_\lambda)-\partial_\tau u^n_\lambda$.
In particular, the iterates converge strongly in $X_2$ to the unique strong solution $u_\lambda$. The bounds on the iterates and residuals and the a posteriori error constant are uniform for $0\le\lambda\le\overline\lambda<\infty$.
\end{theorem}

\subsection{The Additional Entropy Estimate}
The Gibbs variational formula gives
\begin{equation}\label{eq:variational}
 H_\lambda(z)=\sup_\pi\left\{\int_U h(z,\alpha)\,\pi(\dd\alpha)
                         -\lambda\KL(\pi\|\mu)\right\}.
\end{equation}
The maximum is attained by \eqref{eq:gibbs}. At zero temperature the corresponding formula is the maximum over pure actions. Both Hamiltonians are convex in $z$ and globally Lipschitz in $z$, with a common Lipschitz constant independent of $\lambda$. Since $\mu$ is a probability measure,
\begin{equation}\label{eq:linearH}
 \min_\alpha h(z,\alpha)\le H_\lambda(z)\le H_0(z),\qquad
 |H_\lambda(z)|\le C(1+|z|).
\end{equation}

\begin{lemma}[Bound on the Gibbs entropy penalty]\label{lem:entropy}
Let $\pi_\lambda(z)$ be the Gibbs policy. For finite $U$, put $\mu_{\min}=\min_{\alpha\in U}\mu(\{\alpha\})$. Then
\begin{equation}\label{eq:finiteentropy}
 0\le k_\lambda(\pi_\lambda(z))\le H_0(z)-H_\lambda(z)
 \le\lambda\log(1/\mu_{\min}).
\end{equation}
Under \eqref{eq:mass}, for $0<\lambda\le\overline\lambda$,
\begin{equation}\label{eq:logentropy}
 0\le k_\lambda(\pi_\lambda(z))\le H_0(z)-H_\lambda(z)
 \le C\lambda\left[1+\log\left(1+\frac{1+|z|}{\lambda}\right)\right].
\end{equation}
In particular, for each $\varepsilon>0$,
\begin{equation}\label{eq:absorb}
 k_\lambda(\pi_\lambda(z))\le\varepsilon|z|+C_{\varepsilon,\overline\lambda}.
\end{equation}
All estimates are uniform in $(\tau,x)$.
\end{lemma}
\begin{proof}
The identity $k_\lambda(\pi_\lambda(z))=\int h(z,\alpha)\dd\pi_\lambda-H_\lambda(z)$ gives the first inequality by $\int h\dd\pi_\lambda\le H_0$. For finite $U$, retain in the partition function an action attaining $H_0$. For continuous $U$, let $\alpha_*$ maximize $h$ and choose $C>0$ so that $L_z=C(1+|z|)$ bounds its Lipschitz constant in the control variable. Integrating on $B_r(\alpha_*)\cap U$ gives
\[
 H_\lambda(z)\ge H_0(z)-L_zr+\lambda\log(c_\mu r^m).
\]
Choose $r=\min\{r_\mu,\lambda/L_z\}$, adjusting fixed constants if necessary. This proves \eqref{eq:logentropy}. To make the uniformity in temperature explicit, for $s\ge0$,
\[
 \sup_{0<\lambda\le\overline\lambda}
 \lambda\log\left(1+\frac{1+s}{\lambda}\right)
 \le C_{\overline\lambda}\bigl(1+\log(1+s)\bigr).
\]
Indeed the expression inside the supremum is increasing in $\lambda$, and its value at $\overline\lambda$ has the displayed bound. Since $\log(1+s)\le\varepsilon s+C_\varepsilon$, rescaling $\varepsilon$ proves \eqref{eq:absorb}.
\end{proof}

\begin{proof}[Proof of Theorem~\ref{thm:entropy}]
The Gibbs formula defines a measurable kernel and preserves the coefficient bounds under averaging. Lemmas~\ref{lem:inverse} and \ref{lem:entropy}, after subtracting $\widetilde g$, give
\[
 \norm{u^{n+1}_\lambda}_{X_{p_*}}
 \le C\bigl(1+C_{\varepsilon,\overline\lambda}
              +\varepsilon\norm{u^n_\lambda}_{X_{p_*}}\bigr).
\]
Choose $\varepsilon$ so that $C\varepsilon<1/2$. Induction proves existence of every evaluation and a bound uniform in $n$ and $0\le\lambda\le\overline\lambda$. For finite actions, the bounded penalty in \eqref{eq:finiteentropy} gives this directly.

For $n\ge1$, optimality of $\pi_n$ and evaluation of $\pi_{n-1}$ give
\[
 H_\lambda(Ju^n)
 =\ell^{\pi_n}u^n+f^{\pi_n}-k_\lambda(\pi_n)
 \ge\ell^{\pi_{n-1}}u^n+f^{\pi_{n-1}}-k_\lambda(\pi_{n-1})
 =u^n_\tau.
\]
Thus $R_{n,\lambda}\ge0$ and
$L^{\pi_n}(u^{n+1}-u^n)=R_{n,\lambda}$, so the iterates are monotone from $n=1$.
Since the penalty is nonnegative, the upper barrier
$\norm{g}_\infty+M_f\tau$ still applies. Boundedness of $Ju^0$ and Lemma~\ref{lem:entropy} give a common lower bound $-K_0$ on the first source
$f^{\pi_0}-k_\lambda(\pi_0)$. Comparison with
$-\norm{g}_\infty-K_0\tau$ bounds $u^1$ below, and monotonicity bounds every later iterate below. Equation \eqref{eq:linearH} then gives the uniform $L^{p_*}$ residual bound.

The stability argument is unchanged. For any $v,w\in\mathcal X_2$, let $\pi_v,\pi_w$ be their Gibbs policies. The variational formula gives
\[
 \ell^{\pi_w}(v-w)
 \le H_\lambda(Jv)-H_\lambda(Jw)
 \le\ell^{\pi_v}(v-w).
\]
The construction in Lemma~\ref{lem:secant} therefore gives \eqref{eq:nonlinearstability} for $F_\lambda$, uniformly in $\lambda$.

For each fixed temperature, bounded monotonicity implies $u^{n+1}-u^n\to0$ in $L^2$. Lemma~\ref{lem:residual} now gives convergence of $R_{n,\lambda}$ in measure and in $L^2$. The last paragraph of the proof of Theorem~\ref{thm:main}, using the preceding stability estimate, gives convergence, existence, uniqueness, and the a posteriori bound.
\end{proof}

\subsection{Policy Convergence and Temperature Limits}\label{sec:temperature}
\begin{theorem}[Temperature limits]\label{thm:temperature}
Assume both Assumptions~\ref{ass:coeff} and \ref{ass:actions}, and use the same smooth initial guess at every temperature. The bounds in Theorem~\ref{thm:entropy} are uniform for $0\le\lambda\le\overline\lambda<\infty$, and
\begin{equation}\label{eq:uniformiteration}
 \lim_{n\to\infty}\ \sup_{0\le\lambda\le\overline\lambda}
 \norm{u^n_\lambda-u_\lambda}_{X_2}=0.
\end{equation}
For $0<\lambda\le1$,
\begin{equation}\label{eq:zerotemp}
 \norm{u_\lambda-u_0}_{X_2}\le
 \begin{cases}
 C\lambda\log(1/\mu_{\min}),& U\text{ finite},\\
 C\lambda(1+|\log\lambda|),&\text{under \eqref{eq:mass}},
 \end{cases}
\end{equation}
Consequently $u^n_{\lambda_n}\to u_0$ in $X_2$ whenever $n\to\infty$ and $\lambda_n\to0$. This diagonal limit is taken across the family of fixed-temperature schemes.
\end{theorem}

\begin{proposition}[Entropy identity and policy convergence]\label{prop:KL}
For $\lambda>0$ and $n\ge1$,
\begin{equation}\label{eq:KLresidual}
 R_n=\lambda\KL(\pi_{n-1}\|\pi_n)\quad\text{a.e.}
\end{equation}
For each fixed positive $\lambda$, let $\pi_*=\pi_\lambda(Ju_\lambda)$. Then
\[
 \int_Q\norm{\pi_n-\pi_*}_{\TV}^2\dd\tau\dd x\longrightarrow0,
 \qquad
 \norm{\pi-\pi'}_{\TV}:=\frac12\int_U
 \left|\frac{\dd\pi}{\dd\mu}-\frac{\dd\pi'}{\dd\mu}\right|\dd\mu.
\]
\end{proposition}
\begin{proof}
The logarithm of the new density is $(h(Ju^n,\alpha)-H_\lambda(Ju^n))/\lambda$. Integrate it against $\pi_{n-1}$ and use the previous evaluation equation to obtain \eqref{eq:KLresidual}. For two jets, differentiating the Gibbs density along the segment between them gives
\[
 \norm{\pi_\lambda(z)-\pi_\lambda(z')}_{\TV}
 \le (C/\lambda)|z-z'|.
\]
Indeed the derivative of the log density is the centered action score divided by $\lambda$, and the score difference is bounded uniformly by $C|z-z'|$. Integrating the bound on the density derivative proves the displayed inequality. Strong jet convergence now gives the claim. At zero temperature, ties can prevent convergence of the policies even when the values converge.
\end{proof}

\begin{lemma}[A uniform lower bound for value increments]\label{lem:modulus}
For every $\varepsilon>0$ there is $\omega(\varepsilon)>0$, independent of $n\ge1$ and $0\le\lambda\le\overline\lambda$, such that
\begin{equation}\label{eq:modulus}
 \norm{R_{n,\lambda}}_2\ge\varepsilon
 \quad\Longrightarrow\quad
 \norm{u^{n+1}_\lambda-u^n_\lambda}_1\ge\omega(\varepsilon).
\end{equation}
\end{lemma}
\begin{proof}
Use the common bound $\norm{R_{n,\lambda}}_{p_*}\le M_R$. Choose $\gamma=\varepsilon/(2|Q|)^{1/2}$. H\"older's inequality shows that $\norm{R_{n,\lambda}}_2\ge\varepsilon$ implies
\begin{equation}\label{eq:masslower}
 |\{R_{n,\lambda}>\gamma\}|
 \ge\left(\frac{\varepsilon^2}{2M_R^2}\right)^{p_*/(p_*-2)}=:m_\varepsilon>0.
\end{equation}
If the antecedent is possible, the right side is no larger than $|Q|$. Remove a boundary layer of measure less than $m_\varepsilon/2$, cover the remainder by finitely many fixed interior cylinders, and apply the growth argument of Lemma~\ref{lem:residual}. Its constants can be minimized over the finite cover, giving $\norm{z_{n,\lambda}}_2\ge c_\varepsilon>0$.

For $\delta=u^{n+1}_\lambda-u^n_\lambda\ge0$, the correction used in \eqref{eq:corrected}, the uniform inverse, and \eqref{eq:gradientincrement} give
\[
 \norm{z_{n,\lambda}}_2\le C(\norm{\delta}_2+\norm{\delta}_2^{1/2}).
\]
Since $0\le\delta\le2M_\infty$, $\norm{\delta}_2\le(2M_\infty\norm{\delta}_1)^{1/2}$. A sufficiently small $L^1$ increment would contradict the positive lower bound for $z$. This proves \eqref{eq:modulus} with a strictly positive, possibly very small, modulus.
\end{proof}

\begin{proof}[Proof of Theorem~\ref{thm:temperature}]
Theorem~\ref{thm:entropy}, including the unregularized bounds at $\lambda=0$, supplies estimates uniform on the stated temperature interval. Put $e_{n,\lambda}=u_\lambda-u^n_\lambda\ge0$ for $n\ge1$. These errors decrease pointwise. If $\norm{e_{N,\lambda}}_2\ge\varepsilon$, then for every $1\le n\le N$, \eqref{eq:posterior} implies $\norm{R_{n,\lambda}}_2\ge\varepsilon/C$. Lemma~\ref{lem:modulus} and telescoping yield
\[
 N\omega(\varepsilon/C)
 \le\sum_{n=1}^N\norm{u^{n+1}_\lambda-u^n_\lambda}_1
 =\norm{u^{N+1}_\lambda-u^1_\lambda}_1\le2M_\infty|Q|.
\]
Hence $e_{n,\lambda}\to0$ in $L^2$, uniformly in temperature. Since
$0\le u^{n+1}_\lambda-u^n_\lambda\le e_{n,\lambda}$,
the increments tend to zero uniformly in $L^1$. Given $\varepsilon>0$, they are eventually smaller than $\omega(\varepsilon/2)$ for every temperature. The contrapositive of \eqref{eq:modulus} gives $\sup_\lambda\norm{R_{n,\lambda}}_2\le\varepsilon/2<\varepsilon$. The a posteriori estimate proves \eqref{eq:uniformiteration}.

For the zero-temperature limit set $G_\lambda=H_0(Ju_0)-H_\lambda(Ju_0)\ge0$. Since $F_\lambda(u_0)=G_\lambda$ and $F_\lambda(u_\lambda)=0$, the entropy stability estimate proved in Theorem~\ref{thm:entropy} gives
\[
 \norm{u_\lambda-u_0}_{X_2}\le C\norm{G_\lambda}_2.
\]
For finite actions use \eqref{eq:finiteentropy}. In the continuous case, \eqref{eq:logentropy}, $\lambda\le1$, and
\[
 \log\left(1+\frac{1+|Ju_0|}{\lambda}\right)
 \le |\log\lambda|+\log(2+|Ju_0|)
\]
give \eqref{eq:zerotemp}, because $Ju_0\in L^2$ and $Q$ is bounded. The sign of the same secant equation and positivity of the inverse also give $u_\lambda\le u_0$. Combining \eqref{eq:uniformiteration} and \eqref{eq:zerotemp} proves the diagonal limit.
\end{proof}

\begin{corollary}[Error decomposition and intermediate norms]\label{cor:combined}
Under the hypotheses of Theorem~\ref{thm:temperature}, for $n\ge1$ and $0<\lambda\le1$,
\[
 \norm{u^n_\lambda-u_0}_{X_2}
 \le C\norm{R_{n,\lambda}}_2+
 \begin{cases}
 C\lambda\log(1/\mu_{\min}),&U\text{ finite},\\
 C\lambda(1+|\log\lambda|),&\text{under \eqref{eq:mass}}.
 \end{cases}
\]
Moreover, for every $2\le q<p_*$,
\[
 \lim_{n\to\infty}\sup_{0\le\lambda\le\overline\lambda}
 \norm{u^n_\lambda-u_\lambda}_{X_q}=0.
\]
\end{corollary}
\begin{proof}
The first estimate follows from \eqref{eq:posterior}, \eqref{eq:zerotemp}, and the triangle inequality. For $2<q<p_*$, choose $s\in(0,1)$ with $1/q=(1-s)/2+s/p_*$. Interpolation of each derivative in the $X_q$ norm gives
\[
 \norm{u^n_\lambda-u_\lambda}_{X_q}
 \le C\norm{u^n_\lambda-u_\lambda}_{X_2}^{1-s}
       \norm{u^n_\lambda-u_\lambda}_{X_{p_*}}^s.
\]
The last factor is uniformly bounded by the estimates for the iterates and weak lower semicontinuity for the limits. Theorem~\ref{thm:temperature} proves the claim; $q=2$ is already included there.
\end{proof}
The first bound separates iteration error from regularization error.

\section{Quadratic Convergence under Additional Assumptions}\label{sec:quadratic}
Howard iteration has the Newton form
\begin{equation}\label{eq:newton}
 L^{\alpha_n}(u^{n+1}-u^n)=-F(u^n).
\end{equation}
A quadratic rate follows from a quadratic remainder estimate in a norm compatible with the linear inverse. Related function-space considerations appear in \cite{ItoReisingerZhang2021,SmearsSuli2014}. We give two one-dimensional results without entropy: a small-data result in a parabolic H\"older space, and a separable family with a rate in the original Sobolev norm $X_2$.

\subsection{A Small-Data Result in a Parabolic H\"older Space}
Let $I=(0,\ell)$, $Q=(0,T)\times I$, and fix $0<\theta<1$. Write
\[
 \mathcal Y=C^{\theta/2,\theta}(\overline Q),\qquad
 \mathcal H=C^{1+\theta/2,2+\theta}(\overline Q),
\]
using time-first notation for the parabolic H\"older spaces. Choose standard equivalent norms with $\norm{v_{xx}}_{\mathcal Y}\le\norm{v}_{\mathcal H}$, and let $C_{\rm alg}\ge1$ satisfy
\begin{equation}\label{eq:holderproduct}
 \norm{hk}_{\mathcal Y}\le C_{\rm alg}\norm{h}_{\mathcal Y}\norm{k}_{\mathcal Y}.
\end{equation}
Set
\[
 \mathcal H_0=\{v\in\mathcal H:v|_{\partial_pQ}=0\},\qquad
 \mathcal Y_0=\{h\in\mathcal Y:h(0,0)=h(0,\ell)=0\}.
\]
For $a_0>0$, the heat operator $\mathcal L_0=\partial_\tau-a_0\partial_{xx}$ is an isomorphism from $\mathcal H_0$ onto $\mathcal Y_0$, by the classical parabolic Schauder theory \cite[Chapter IV]{LSU1968}. The two corner conditions are the first compatibility conditions for zero initial and lateral data. Denote
\begin{equation}\label{eq:heatconstant}
 B=\norm{\mathcal L_0^{-1}}_{\mathcal Y_0\to\mathcal H_0}<\infty.
\end{equation}
This reference constant depends on $a_0,T,\ell,\theta$ and the chosen norms.

\begin{proposition}[Quadratic convergence from zero initial guess]\label{prop:quadratic}
Let $\nu\ne0$, $r_0>0$, $M>0$, $a_0>|\nu|M$, and $f\in\mathcal Y_0$. Consider
\begin{equation}\label{eq:quadraticmodel}
 u_\tau=\max_{|\alpha|\le M}
 \left\{(a_0+\nu\alpha)u_{xx}+f(\tau,x)-\frac{r_0}{2}\alpha^2\right\},
 \qquad u|_{\partial_pQ}=0.
\end{equation}
Put $\kappa=\nu^2/r_0$ and $K=C_{\rm alg}\kappa$. Suppose there is $R>0$ such that
\begin{equation}\label{eq:quadraticsmall}
 B\norm{f}_{\mathcal Y}\le\frac R4,\qquad
 BKR\le\frac12,\qquad
 \frac{|\nu|R}{r_0}\le\frac M2.
\end{equation}
Then \eqref{eq:quadraticmodel} has a solution $u_*\in\mathcal H_0$ with $\norm{u_*}_{\mathcal H}\le R/2$, equal to the unique strong solution of Theorem~\ref{thm:main}. Howard iteration starting at $u^0=0$ stays in the closed $\mathcal H_0$ ball of radius $R$. Its greedy policies and evaluations are
\begin{align}
 \alpha_n&=\frac\nu{r_0}u^n_{xx},\qquad |\alpha_n|\le M/2,\label{eq:quadraticpolicy}\\
 \bigl[\partial_\tau-(a_0+\kappa u^n_{xx})\partial_{xx}\bigr]u^{n+1}
 &=f-\frac\kappa2(u^n_{xx})^2,\qquad u^{n+1}|_{\partial_pQ}=0.\label{eq:quadraticeval}
\end{align}
For $E_n=\norm{u_*-u^n}_{\mathcal H}$,
\begin{equation}\label{eq:quadraticrate}
 E_{n+1}\le BK E_n^2,\qquad
 E_n\le\frac{(BK E_0)^{2^n}}{BK},\qquad BK E_0\le\frac14.
\end{equation}
In particular, the convergence is Q-quadratic in $\mathcal H$.
\end{proposition}
\begin{proof}
For $v\in\mathcal H_0$ with $\norm{v}_{\mathcal H}\le R$, the maximizer in \eqref{eq:quadraticmodel} is $(\nu/r_0)v_{xx}$ and lies in $[-M/2,M/2]$. On this ball the Bellman operator therefore equals
\[
 \mathcal F(v)=\mathcal L_0v-f-\frac\kappa2(v_{xx})^2.
\]
Since $v(0,x)=0$, continuity of its spatial derivatives gives $v_{xx}(0,x)=0$. Thus $(v_{xx})^2\in\mathcal Y_0$ and all the following equations satisfy the required corner compatibility.

First construct a solution. The map
\[
 \Phi(v)=\mathcal L_0^{-1}\left[f+\frac\kappa2(v_{xx})^2\right]
\]
maps the closed radius-$R$ ball into the radius-$R/2$ ball, because
\[
 \norm{\Phi(v)}_{\mathcal H}
 \le B\norm{f}_{\mathcal Y}+\frac{BK}{2}R^2\le\frac R2.
\]
Using \eqref{eq:holderproduct} on the difference of two squares also gives
\[
 \norm{\Phi(v)-\Phi(w)}_{\mathcal H}
 \le BKR\norm{v-w}_{\mathcal H}\le\tfrac12\norm{v-w}_{\mathcal H}.
\]
The contraction theorem provides a fixed point $u_*$ with norm at most $R/2$. Its maximizing control remains interior, so it solves the original maximization problem. It belongs to $X_2$, so uniqueness in Theorem~\ref{thm:main} identifies the two solutions.

Next, for every $v$ in the radius-$R$ ball, let
\[
 \mathcal L_v=\mathcal F'(v)
 =\partial_\tau-(a_0+\kappa v_{xx})\partial_{xx}.
\]
The operator $w\mapsto\mathcal L_0^{-1}(\kappa v_{xx}w_{xx})$ on $\mathcal H_0$ has norm at most $BKR\le1/2$. Its Neumann series yields
\begin{equation}\label{eq:holderinverse}
 \norm{\mathcal L_v^{-1}}_{\mathcal Y_0\to\mathcal H_0}\le2B.
\end{equation}
The coefficients are also uniformly elliptic: $a_0+\kappa v_{xx}=a_0+\nu\alpha(v)\ge a_0-|\nu|M/2>0$. Equation \eqref{eq:quadraticeval} is the Newton identity
$\mathcal F'(u^n)(u^{n+1}-u^n)=-\mathcal F(u^n)$.

Whenever $u^n$ lies in the ball, subtracting its evaluation from the equation for $u_*$ gives the identity
\begin{equation}\label{eq:quadraticremainder}
 \mathcal L_{u^n}(u_*-u^{n+1})
 =\frac\kappa2\bigl[(u_*-u^n)_{xx}\bigr]^2.
\end{equation}
Equations \eqref{eq:holderproduct} and \eqref{eq:holderinverse} imply $E_{n+1}\le BK E_n^2$. Initially $E_0\le R/2$ and $BK E_0\le1/4$. Inductively, $E_n\le R/2$ implies $E_{n+1}\le E_n/4$ and
$\norm{u^{n+1}}_{\mathcal H}\le\norm{u_*}_{\mathcal H}+E_{n+1}\le R$.
This proves that every evaluation is well defined and that no policy leaves the stated class. Iterating $BK E_{n+1}\le(BK E_n)^2$ proves \eqref{eq:quadraticrate}.
\end{proof}

The smallness conditions are nonempty and refer only to the model data and the fixed heat inverse. For $f\ne0$, one may take $R=4B\norm{f}_{\mathcal Y}$ and require
\[
 8C_{\rm alg}B^2\frac{\nu^2}{r_0}\norm{f}_{\mathcal Y}\le1,
 \qquad
 8B\frac{|\nu|}{r_0}\norm{f}_{\mathcal Y}\le M.
\]
For example, $f(\tau,x)=\varepsilon\tau x(\ell-x)$ satisfies these conditions for sufficiently small $|\varepsilon|$, with the other parameters fixed. The use of H\"older spaces is essential to this proof: the product estimate controls the square of the Hessian error in the same forcing space in which the linear inverse is bounded. The continuous embedding $\mathcal H\hookrightarrow X_2$ also transfers the doubly exponential upper bound in \eqref{eq:quadraticrate} to the $X_2$ norm.

\subsection{Quadratic Convergence in the Sobolev Norm}
\begin{proposition}[A separable family]\label{prop:quadraticXtwo}
Let $a_0,r_0,M,\eta>0$, $\nu\ne0$, and $a_0>|\nu|M$. Set
\[
 \kappa=\frac{\nu^2}{r_0},\qquad
 \delta=\sqrt{a_0^2-2\kappa\eta},\qquad
 \overline P=\frac{a_0-\delta}{\kappa},
\]
and assume $a_0^2>2\kappa\eta$ and $|\nu|\overline P/r_0<M$.
On $Q=(0,T)\times(0,\pi)$ consider
\begin{equation}\label{eq:modeBellman}
 u_\tau=\max_{|\alpha|\le M}
 \left\{(a_0+\nu\alpha)u_{xx}
       +\left(\eta-\frac{r_0}{2}\alpha^2\right)\sin x\right\},
 \qquad u|_{\partial_pQ}=0.
\end{equation}
For every finite $T$, Howard iteration from $u^0=0$ converges to the unique strong solution $u_*$ and satisfies
\begin{equation}\label{eq:modeXrate}
 \norm{u_*-u^{n+1}}_{X_2}
 \le C_T\norm{u_*-u^n}_{X_2}^{\,2},\qquad n\ge0,
\end{equation}
where one may take
\[
 C_T=\frac{\kappa T^{3/2}}{\sqrt{2\pi}}
             \left(\frac{a_0+3}{\delta}+1\right)
\]
for the $X_2$ norm fixed in Section~\ref{sec:setting}.
\end{proposition}
\begin{proof}
The scalar equation
\[
 P'=-a_0P+\eta+\frac{\kappa}{2}P^2,\qquad P(0)=0
\]
has a global solution in $[0,\overline P]$: its right side is positive at zero and vanishes at $\overline P$. It follows that
$u_*(\tau,x)=P(\tau)\sin x$ solves \eqref{eq:modeBellman}, with interior maximizing control $-\nu P/r_0$.

Starting with $P_0=0$, define
\begin{equation}\label{eq:modeeval}
 P_{n+1}'=(-a_0+\kappa P_n)P_{n+1}
                    +\eta-\frac{\kappa}{2}P_n^2,
 \qquad P_{n+1}(0)=0.
\end{equation}
If $0\le P_n\le\overline P$, the right side at $P_{n+1}=0$ is at least
$\eta-\kappa\overline P^2/2=\delta\overline P>0$.
At $P_{n+1}=\overline P$, it equals
$-\kappa(\overline P-P_n)^2/2\le0$.
Scalar comparison gives $0\le P_{n+1}\le\overline P$. Since $\sin x>0$ in $(0,\pi)$, the greedy control for $P_n(\tau)\sin x$ is $-\nu P_n/r_0$, and \eqref{eq:modeeval} is its policy evaluation. Uniqueness of each linear evaluation proves that the Howard sequence itself satisfies $u^n=P_n\sin x$. Theorem~\ref{thm:main} therefore gives convergence to $u_*$.

Let $e_n=P-P_n$. Subtraction gives
\begin{equation}\label{eq:modeerror}
 e_{n+1}'=(-a_0+\kappa P_n)e_{n+1}
                         +\frac{\kappa}{2}e_n^2,\qquad e_{n+1}(0)=0.
\end{equation}
Since $\delta\le a_0-\kappa P_n\le a_0$, variation of constants gives
\begin{align}
 \norm{e_{n+1}}_\infty
 &\le\frac{\kappa}{2\delta}\norm{e_n}_\infty^2,\label{eq:modesuprate}\\
 \norm{e_{n+1}'}_\infty
 &\le\frac{\kappa}{2}\left(\frac{a_0}{\delta}+1\right)
                    \norm{e_n}_\infty^2.\label{eq:modetimerate}
\end{align}
Here the norms are on $(0,T)$. Put $s=\sqrt{\pi/2}$. Direct integration of $\sin^2x$ and $\cos^2x$ yields
\[
 \norm{e(\tau)\sin x}_{X_2}
 =s\bigl(3\norm{e}_{L^2(0,T)}+\norm{e'}_{L^2(0,T)}\bigr).
\]
Furthermore, $e_n(0)=0$ gives
$\norm{e_n}_\infty\le\sqrt T\norm{e_n'}_2
\le(\sqrt T/s)\norm{u_*-u^n}_{X_2}$.
Combining these facts with \eqref{eq:modesuprate}--\eqref{eq:modetimerate} proves \eqref{eq:modeXrate} with the displayed constant.
\end{proof}

\section{Examples}\label{sec:examples}
The first example quantifies the quadratic rate. The remaining examples illustrate large diffusion variation, discontinuous improvement, and multidimensional extensions.

\begin{example}[Quadratic error bounds at a fixed horizon]\label{ex:quadraticXtwo}
In Proposition~\ref{prop:quadraticXtwo}, take
\[
 T=1,\qquad a_0=2,\qquad \nu=r_0=1,\qquad
 \eta=\tfrac14,\qquad M=\tfrac12.
\]
Then $\delta=\sqrt{7/2}$, $\overline P=2-\delta<1/2$, and the diffusion coefficient ranges over $[3/2,5/2]$. The limiting amplitude is
\[
 P(\tau)=
 \frac{(2-\delta)(1-e^{-\delta\tau})}
      {1-\frac{2-\delta}{2+\delta}e^{-\delta\tau}}.
\]
This follows by separation of variables in the scalar equation in the preceding proof. The first evaluation is
$P_1(\tau)=\tfrac18(1-e^{-2\tau})$; later evaluations solve \eqref{eq:modeeval}.

Let $\mathcal E_n=\norm{P-P_n}_{L^\infty(0,1)}$. Since $P$ increases,
$\mathcal E_0=P(1)$. Equation \eqref{eq:modesuprate} gives the analytical bound
\begin{equation}\label{eq:modebenchmark}
 \mathcal E_n\le B_n,\qquad
 B_n=2\delta\left(\frac{P(1)}{2\delta}\right)^{2^n}.
\end{equation}
The values, spatial curvatures, and controls have the same supremum error in this example:
\[
 \norm{u_*-u^n}_\infty
 =\norm{u_{*,xx}-u^n_{xx}}_\infty
 =\norm{\alpha_*-\alpha_n}_\infty
 =\mathcal E_n.
\]
Thus the bound describes both value approximation and policy improvement. Table~\ref{tab:bounds} lists the corresponding upper bounds.
\begin{table}[htbp]
\caption{Analytical upper bounds for the separable example.}\label{tab:bounds}
\centering
\renewcommand{\arraystretch}{1.2}
\setlength{\tabcolsep}{5pt}
\begin{tabular}{cr}
\toprule
$n$ & $B_n$\\
\midrule
$0$ & $1.0984\times10^{-1}$\\
$1$ & $3.2247\times10^{-3}$\\
$2$ & $2.7792\times10^{-6}$\\
$3$ & $2.0642\times10^{-12}$\\
$4$ & $1.1388\times10^{-24}$\\
\bottomrule
\end{tabular}
\end{table}
The entries are rounded analytical upper bounds computed from \eqref{eq:modebenchmark}. The $X_2$ estimate follows separately from \eqref{eq:modeXrate}, which also controls the time derivative.
\end{example}

\begin{example}[Risk exposure before exit and maturity]\label{ex:scalar}
Take $\Omega=(0,1)$, $U=\{1,2\}$, $a^1=1$, $a^2=100$, $b^1=b^2=0$, $c^1=c^2=0$, $f^1=f^2=1$, and $g(x)=\sin(\pi x)+\tfrac14\sin(3\pi x)$. The equation is
\[
 u_\tau=\max\{u_{xx},100u_{xx}\}+1.
\]
With the initial guess $u^0=\widetilde g$,
$g''(x)=-\pi^2\sin(\pi x)[31/4-9\sin^2(\pi x)]$.
The first greedy policy therefore selects diffusivity $100$ on $(x_*,1-x_*)$ and diffusivity $1$ outside that interval, where $x_*=\pi^{-1}\arcsin(\sqrt{31}/6)$. Either action can be selected at the two interfaces. This gives a discontinuous second-order coefficient in the first policy evaluation, despite the smooth initial guess. Theorem~\ref{thm:main} applies for every finite $T$.

A concrete associated control objective is to manage the risk exposure of a normalized performance index,
\[
\begin{gathered}
 dX_s=\sqrt{2a^{\alpha_s}}\,dB_s,\qquad
 \zeta=\inf\{s\ge t:X_s\notin(0,1)\}\wedge T,\\
 J^\alpha(t,x)=\E_{t,x}\bigl[\zeta-t+g(X_\zeta)\bigr].
\end{gathered}
\]
For a given progressively measurable control process, the state is the corresponding It\^o integral up to exit. The two actions select instantaneous variances $2$ and $200$. The controller earns one unit per unit time while the system remains in its operating interval and receives the terminal performance reward if it survives to maturity; early exit gives no terminal reward. The nonconcave $g$ has two preferred interior levels. In the first improvement, the convex middle part of $g$ favors larger variance, whereas its concave outer parts favor smaller variance.
\end{example}

\begin{example}[Nonaffine continuous actions and entropy]
Let $U=[-1,1]$ with normalized Lebesgue measure and set
\[
 A^\alpha=(1+99\alpha^2)S(\tau,x),\quad
 b^\alpha=b_0(\tau,x)+\alpha b_1(\tau,x),\quad
 f^\alpha=f_0(\tau,x)-q_0\alpha^2,
\]
where the displayed fields are bounded and measurable, $q_0\ge0$, $S$ is smooth uniformly positive, and $c^\alpha=c_0(\tau,x)$ is bounded nonnegative and independent of the action. All coefficients are uniformly Lipschitz in $\alpha$, so Assumption~\ref{ass:actions} holds. At a jet $z=(r,p,M)$, the Gibbs density with respect to $\mu$ has the form
\[
 \frac{\dd\pi_\lambda(z)}{\dd\mu}(\alpha)
 \propto\exp\!\left\{\frac{(99S:M-q_0)\alpha^2+(b_1\cdot p)\alpha}{\lambda}\right\},
\]
with the density normalized to integrate to one. The quadratic action coefficient changes sign as the Hessian varies. Theorem~\ref{thm:entropy} therefore applies to both convex and concave action dependence in this example.
\end{example}

\begin{example}[Controlled anisotropy]
On a smooth bounded convex domain in $\R^2$, take $U=[0,2\pi]$ and set
\[
 A^\alpha=a^\alpha I+\varepsilon
 \begin{pmatrix}\cos\alpha&\sin\alpha\\\sin\alpha&-\cos\alpha\end{pmatrix},
 \qquad a^\alpha=1+99\sin^2\alpha.
\]
Here $|E^\alpha|_F=\sqrt2\varepsilon$. Thus $0<\sqrt2\varepsilon<1$ is an explicit sufficient condition. The action changes the eigendirections as well as the amplitude. With normalized Lebesgue measure this example also satisfies the entropy assumptions, for bounded action-Lipschitz lower-order coefficients. The perturbation is small only in its non-scalar part.
\end{example}

\section{Conclusions}\label{sec:scope}
Howard iteration converges strongly for the finite-horizon Bellman equations considered here, even when policy improvement produces merely measurable evaluation coefficients. In one dimension, the controlled diffusion may vary arbitrarily within fixed ellipticity bounds. The same argument gives entropy-regularized convergence and strong zero-temperature limits.

For more general diffusion matrices, the main question is whether uniform linear estimates above the square-integrability exponent remain valid under the coefficient averages and approximations used in the proof. Rates for rough switching policies and convergence of fully discrete implementations are further questions.

\section*{Acknowledgements}
This work was supported in part by the National Natural Science Foundation of China (NSFC) under Grant Nos.~12271274 and~12571514.

\section*{Declarations}
\paragraph{Conflict of interest} The author declares no competing financial or non-financial interests.
\paragraph{Data availability} No datasets were generated or analyzed during this study.

\begingroup
\small
\bibliographystyle{spmpsci-local}
\bibliography{references}
\endgroup
\end{document}